\documentclass[11pt,reqno]{amsart}

\usepackage{amsmath, amsthm, amssymb}
\usepackage{amsfonts}
\usepackage[ansinew]{inputenc}
\usepackage[dvips]{epsfig}
\usepackage{graphicx}
\usepackage[english]{babel}

\usepackage{cite}
\usepackage{graphicx}
\usepackage{amscd}
\usepackage{bm}
\usepackage{enumerate}

\usepackage{verbatim}
\usepackage{hyperref}
\usepackage{amstext}
\usepackage{latexsym}
\theoremstyle{plain}
\newtheorem{theorem}{Theorem}[section]

\newtheorem{corollary}[theorem]{Corollary}
\newtheorem{proposition}[theorem]{Proposition}
\newtheorem{lemma}[theorem]{Lemma}
\newtheorem{remark}[theorem]{Remark}
\newtheorem{example}[theorem]{Example}

\numberwithin{theorem}{section}
\numberwithin{equation}{section}

\newcommand{\average}{{\mathchoice {\kern1ex\vcenter{\hrule height.4pt
width 6pt depth0pt} \kern-9.7pt} {\kern1ex\vcenter{\hrule
height.4pt width 4.3pt depth0pt} \kern-7pt} {} {} }}

\def\R{\mathbb{R}}

\def\div{\text{div}}

\renewcommand{\a }{\alpha }

\renewcommand{\d}{\delta }

\newcommand{\D }{\Delta }

\newcommand{\e }{\varepsilon }

\newcommand{\G }{\Gamma}
\renewcommand{\l }{\lambda }

\newcommand{\n }{\nabla }
\newcommand{\vp }{\varphi }

\newcommand{\s }{\sigma }

\newcommand{\z }{\zeta}

\renewcommand{\O }{\Omega }

\newcommand{\ov}{\overline}

\newcommand{\be}{\begin{equation}}
\newcommand{\ee}{\end{equation}}

\newcommand{\de}{\partial}

\newcommand{\al}{\alpha}

 \usepackage{mathrsfs}

\newcommand{\calH }{\mathcal{H}}

\newcommand{\calE }{\mathcal{E}}

\newcommand{\N}{\mathbb{N}}

\newcommand{\cE}{{\mathcal E}}

\newcommand{\cH}{{\mathcal H}}

\newcommand{\cL}{{\mathcal L}}

\newcommand{\cX}{{\mathcal X}}
\newcommand{\cY}{{\mathcal Y}}

\newcommand{\dist}{{\rm dist}}

\newcommand{\eps}{\varepsilon}

\DeclareMathOperator{\id}{id}

\renewcommand{\epsilon}{\varepsilon}

\newcommand{\Ds}{ (-\D)^s}

\begin{document}
\title{A generalized fractional Pohozaev identity and applications}

 \author[]
{Sidy Moctar Djitte${}^{1,2}$, Mouhamed Moustapha Fall${}^1$, Tobias Weth${}^2$}


\address{${}^1$African Institute for Mathematical Sciences in Senegal (AIMS Senegal), 
KM 2, Route de Joal, B.P. 14 18. Mbour, S\'en\'egal.}

\address{${}^2$Goethe-Universit\"{a}t Frankfurt, Institut f\"{u}r Mathematik.
Robert-Mayer-Str. 10, D-60629 Frankfurt, Germany.}

\email{djitte@math.uni-frankfurt.de, sidy.m.djitte@aims-senegal.org}
\email{weth@math.uni-frankfurt.de}
\email{mouhamed.m.fall@aims-senegal.org}

 \begin{abstract}
   \noindent 
   We prove a fractional Pohozaev type identity in a generalized framework and discuss its applications. Specifically, we shall consider applications to nonexistence of solutions in the case of supercritical semilinear Dirichlet problems and regarding a Hadamard formula for the derivative of Dirichlet eigenvalues of the fractional Laplacian with respect to domain deformations. We also derive the simplicity of radial eigenvalues in the case of radial bounded domains and apply the Hadamard formula to this case. 
 \end{abstract}

\maketitle
\setcounter{equation}{0}

\section{Introduction}

Let $\O$ be a bounded open set of class $C^{1,1}$ and $s \in (0,1)$. We consider the semilinear fractional Dirichlet problem
\begin{equation}
  \label{eq:gradient-type-intro}
(-\Delta)^s u = f(u) \quad \text{in $\Omega$},\qquad u = 0 \quad \text{on $\R^N \setminus \Omega$.}  
\end{equation}
Here $(-\Delta)^s$ denotes the fractional Laplacian, which, for sufficiently regular functions $\vp$, is pointwisely given by 
$$
\Ds\vp(x)=c_{N,s}\,PV \!\!\int_{\R^N} \frac{\vp(x)-\vp(y)}{|x-y|^{N+2s}}\, dy=\frac{c_{N,s}}{2}\int_{\R^N}\frac{2\vp (x)-\vp(x+y)-\vp(x-y)}{|y|^{N+2s}}\, dy.
$$
with $c_{N,s}= \pi^{-\frac{N}{2}}s 4^s\frac{\Gamma(\frac{N}{2}+s)}{\Gamma(1-s)}$.
Moreover, we assume that
\begin{equation}
  \label{eq:lipschitz-assumption}
\text{$f: \R \to \R$ in (\ref{eq:gradient-type-intro}) is locally Lipschitz,}
\end{equation}
and we let $F \in C^1(\R)$ be defined by $F(t)= \int_0^t f(s)\,ds$.
We consider (\ref{eq:gradient-type-intro}) in weak sense. For this we define
\begin{equation}\label{wsp-d-special-case}
\cH^{s}_{0}(\Omega):=\{u\in H^{s}(\R^N) : u\equiv 0 ~\text{ on\ \ $\R^{N}\setminus \Omega$}\} \subset H^{s}(\R^N).
\end{equation}
Here $H^s(\R^N)$ is the set of those functions $u$ for which ${\cE}(u,u)$, with ${\cE}$ define as in \eqref{bilinear form}, is finite. By definition, a function $u \in \cH^s_0(\Omega) \cap L^\infty(\Omega)$ is a weak solution of (\ref{eq:gradient-type-intro}) if 
$$
{\cE}(u,v) = \int_{\Omega} f(u)v\,dx \qquad \text{for all $v \in \cH^s_0(\Omega)$,}
$$
where 
\begin{equation}\label{bilinear form}
(v,w) \mapsto {\cE}(v,w):=\frac{c_{N,s}}{2}\int_{\R^N}\int_{\R^N}\frac{(v(x)-v(y))(w(x)-w(y))}{|x-y|^{N+2s}}\ dxdy.
\end{equation}
From (\ref{eq:lipschitz-assumption}) and the elliptic regularity theory for weak solutions developed in recent years (see \cite{RX,S}), it follows that every weak solution $u \in \cH^s_0(\Omega) \cap L^\infty(\Omega)$ is contained in the space $C_0^s(\overline \Omega) \cap C^{2s+1-\eps}_{loc}(\Omega)$ for every $\eps \in (0,2s+1)$. Here $C_0^s(\ov\O)=\{u\in C^s(\ov\O):u=0\quad\text{in}\quad\R^N\setminus\O\}$. Moreover, it has been proved
in \cite{RX} that
$$
\text{the function $\psi_u:= \frac{u}{\delta^s}$ extends uniquely to a function in $C^\a(\ov\O)$ for some $\a>0$,}
$$
where, here and in the following, we let $\d(x)= \dist(x,\R^N \setminus \Omega)$ for $x \in \R^N$.

In the seminal paper \cite{RX-Poh}, Ros-Oton and Serra introduced and proved a fractional Pohozaev identity which states that every (weak) solution  of (\ref{eq:gradient-type-intro}) satisfies
\begin{equation}
\label{pohozaev-rosoton-serra}
  \Gamma(1+s)^2 \int_{\partial \Omega} \psi_u^2\, x \cdot \nu\, d\sigma = 2N \int_\Omega F(u)\,dx - (N-2s) \int_\Omega f(u)u \,dx,
\end{equation}
see \cite[Theorem 1.1]{RX-Poh}. Here $\nu$ in (\ref{pohozaev-rosoton-serra}) is the unit outer normal vector field. This identity has proved to be highly relevant in the study of (\ref{eq:gradient-type-intro}). In particular, it yields a nonexistence result for (\ref{eq:gradient-type-intro}) in the case where $\Omega$ is starshaped and $f$ satisfies a supercritical growth condition, see \cite[Corollary 1.3]{RX-Poh}. Somewhat surprisingly, (\ref{pohozaev-rosoton-serra}) is already useful in the linear case $f(u)= \lambda u$, as it gives valuable information on the fractional boundary derivative $\psi_u:= \frac{u}{\delta^s}$ of Dirichlet eigenfunctions of the fractional Laplacian $(-\Delta)^s$. In particular, as we shall see in Section~\ref{eq-eigenvalue-problem} below, it allows to show the simplicity of radial Dirichlet eigenvalues of $(-\Delta)^s$ in the case where $\Omega$ is a ball or an annulus. Moreover, (\ref{pohozaev-rosoton-serra}) has been used recently in \cite{FFTW} to prove the nonradiality of second Dirichlet eigenfunctions of $(-\Delta)^s$ in the case $\Omega= B_1(0)$. Note that these properties are standard in the local case $s=1$, where tools like separation of variables and ODE techniques are available.   

The main purpose of this paper is to present a generalization of the identity (\ref{pohozaev-rosoton-serra}) depending on a given Lipschitz vector field $\cX \in C^{0,1}(\R^N,\R^N)$. We recall that every such vector field is a.e. differentiable on $\R^N$, so its derivative $d\, \cX$ and also $\div\,\cX$ are a.e. well defined on $\R^N$. For every such vector field, we let
\begin{align}
  &K_{\cX}(x,y)\label{def-kernel-K-X}:= \frac{c_{N,s}}{2}\Bigl[\bigl(\div\, \cX(x) + \div\, \cX(y)\bigr)- (N+2s)\frac{ \bigl(\cX(x)-\cX(y)\bigr)\cdot (x-y)}{|x-y|^2}\Bigr]\frac{1}{|x-y|^{N+2s}} 
\end{align}
for $x,y \in \R^N$, $x \not = y$, and we call $K_{\cX}$ the {\em fractional deformation kernel associated with the vector field $\cX$}. We will justify this name further below. Moreover, we denote by $\calE_{\cX}$ the bilinear form associated to the Kernel $K_\cX$, i.e, 
\be\label{eq-K-X-bilinear-form}
\calE_{\cX}(v,w):=\int_{\R^N}\int_{\R^N}(v(x)-v(y))(w(x)-w(y))K_\cX(x,y)\,dxdy\quad \text{for all $v,w\in H^s(\R^N).$}
\ee
Our first main result for problem \eqref{eq:gradient-type-intro} is the following. 
\begin{theorem}
\label{generalized-pohozaev}
Let $u \in \cH^s_0(\Omega) \cap L^\infty(\Omega)$ be a (weak) solution of the problem~(\ref{eq:gradient-type-intro}).
Then we have
\begin{equation}
  \label{eq:V-prime-formula}
  \Gamma(1+s)^2 \int_{\partial \Omega} (\frac{u}{\d^s})^2\,  \cX \cdot \nu \,dx   =
  2\int_{\Omega} F(u) \div \, \cX\,dx - \calE_{\cX}(u,u)\quad \text{for all $\cX\in C^{0,1}(\R^N,\R^N)$}
\end{equation}
with $F(t)= \int_0^t f(s)\,ds$. Here $\nu$ is the outer unit normal to the boundary and $\calE_{\cX}(v,w)$ is defined as in \eqref{eq-K-X-bilinear-form}.
\end{theorem}

Theorem \ref{generalized-pohozaev} is a particular case of the following more general identity. 
\begin{theorem}\label{Generalized-integration-by-parts}
Let $u\in H^s(\R^N)$ such that $u\equiv0$ in $\R^N\setminus\O$. Moreover, assume $\Ds u\in L^\infty(\O)$ if $2s>1$ and $\Ds u\in C^\al_{loc}(\O)\cap L^\infty(\O)$ with  $\al>1-2s$ if $2s\leq 1$. Then we have
\be\label{eq-generalized-integration-by-parts}
2\int_{\O}\n u\cdot\cX\Ds u\,dx=-\G^2(1+s)\int_{\partial\O}\big(\frac{u}{\d^s}\big)^2\cX\cdot\nu\,dx-\calE_{\cX}(u,u),
\ee
for any vector field $\cX\in C^{0,1}(\R^N,\R^N)$.
\end{theorem}
To deduce formula \eqref{eq:V-prime-formula} from \eqref{eq-generalized-integration-by-parts} it simply suffices to use the pointwise identities $\Ds u=f(u)$, $\n F(u)=f(u)\n u$ and to integrate by parts, noting that $F(0)=0$. As noted already above, the regularity assumptions of Theorem \ref{Generalized-integration-by-parts} are satisfied in this case as a consequence of assumption (\ref{eq:lipschitz-assumption}) and
the elliptic regularity theory for weak solutions developed in \cite{RX,S}.
We note that Theorem~\ref{Generalized-integration-by-parts} generalizes \cite[Proposition 1.6]{RX-Poh} where the particular vector field $\cX \equiv \id: \R^N \to \R^N$ is considered. Indeed, in the case $\cX \equiv \id$, we have
$$
\div\, \cX \equiv N\qquad \text{and}\qquad K_{\cX}(x,y) = \frac{c_{N,s}}{2}(N-2s)|x-y|^{-N-2s}\quad \text{for $x,y \in \R^N$,}
$$
so \eqref{eq-generalized-integration-by-parts} reduces to 
$$
2\int_{\O}(x\cdot\n u)\Ds u\,dx=-\G^2(1+s)\int_{\partial\O}\big(\frac{u}{\d^s}\big)^2x\cdot\nu\,dx-(N-2s)\int_{\O}u\Ds u\,dx.
$$
This is the identity stated in \cite[Proposition 1.6]{RX-Poh}. Moreover, for every weak solution of (\ref{eq:gradient-type-intro}) we have 
$$
\calE_{\cX}(u,u)= (N-2s){\cE}(u,u)= (N-2s) \int_{\Omega}f(u)u\,dx
$$
in this case, and therefore (\ref{eq:V-prime-formula}) reduces to (\ref{pohozaev-rosoton-serra}).\\

We also note the following integration-by-parts formula, which is an immediate consequence of Theorem~\ref{Generalized-integration-by-parts}.

\begin{theorem}\label{main-thm-P4}
Let $u,v\in H^s(\R^N)$ be functions with $u\equiv0\equiv v$ in $\R^N\setminus\O$. Moreover, assume $\Ds u,\,\Ds v\in L^\infty(\O)$ if $2s>1$ and $\Ds u,\,\Ds v\in C^\al_{loc}(\O)\cap L^\infty(\O)$ with  $\al>1-2s$ if $2s\leq 1$. Then, for any vector field $\cX\in C^{0,1}(\R^N,\R^N)$, it holds that 
\be\label{eq-integration-by-parts2}
\int_{\O}\n u\cdot \cX\Ds v\,dx=-\int_{\O}\n v\cdot\cX\Ds u\,dx-\G^2(1+s)\int_{\partial\O}\frac{u}{\d^s}\frac{v}{\d^s}\,\cX\cdot\nu\,d\sigma-\calE_{\cX}(u,v).
\ee
\end{theorem}

To deduce this theorem from Theorem~\ref{Generalized-integration-by-parts}, it suffices to apply \eqref{eq-generalized-integration-by-parts} to $u,v$ and $u+v$ and to evaluate the difference $\calE_{\cX}(u+v,u+v)-\calE_{\cX}(v,v)-\calE_{\cX}(u,u)$. We note that Theorem~\ref{main-thm-P4} is stated in \cite[Theorem 1.9]{RX-Poh} in the particular case of constant coordinate vector fields $\cX \equiv e_i$, $i=1,\dots,N$, in which $K_\cX \equiv 0$ and therefore \eqref{eq-integration-by-parts2} reduces to  
$$
\int_{\O}u_{x_i} \Ds v\,dx=-\int_{\O} v_{x_i}  \Ds u\,dx-\G^2(1+s)\int_{\partial\O}\frac{u}{\d^s}\frac{v}{\d^s}\,\nu_i\,d\sigma.
$$

The following corollary of Theorem~\ref{generalized-pohozaev} is devoted again to problem (\ref{eq:gradient-type-intro}) and deals with a class of vector fields leading to the same RHS as in (\ref{pohozaev-rosoton-serra}) (up to a constant).

\begin{corollary}
  \label{sec:ros-oton-generalized}
  Let $\cX \in C^{0,1}(\R^N,\R^N)$, and suppose that
  \begin{equation}
    \label{eq:c-condition}
\bigl(\cX(x)-\cX(y)\bigr) \cdot (x-y)   = c|x-y|^2 \qquad \text{for all $x,y \in \R^N$}
  \end{equation}
with some constant $c \in \R$. Moreover, let $u \in \cH^s_0(\Omega) \cap L^\infty(\Omega)$ be a (weak) solution of the problem~(\ref{eq:gradient-type-intro}). Then we have
\begin{equation}
  \label{eq:V-prime-formula-cor-1}
  \Gamma(1+s)^2 \int_{\partial \Omega} \psi_u^2\,  \cX \cdot \nu \,dx   =
  c \Bigl( 2N \int_\Omega F(u)\,dx - (N-2s) \int_\Omega f(u)u \,dx \Bigr).
  \end{equation}
\end{corollary}

\begin{remark}
\label{remark-intro-1}
{\rm  It is easy to see that condition (\ref{eq:c-condition}) is equivalent to
  \begin{equation}
    \label{eq:c-condition-variant}
\bigl(d\cX(y)h\bigr) \cdot h =  c|h|^2 \qquad \text{for a.e. $y \in \R^N$ and every $h \in \R^N$.}
  \end{equation}
  Applying (\ref{eq:c-condition-variant}) to the coordinate vectors $e_1,\dots,e_N \in \R^N$, we deduce that $\div \,\cX \equiv c N$ a.e. on $\R^N$.
We note that condition (\ref{eq:c-condition}) is satisfied if
\begin{equation}
  \label{eq:c-cond-sufficient-con}
  x \mapsto \cX(x) = c x + \cY(x) + v,
\end{equation}
where $v \in \R^N$ is a constant vector and $\cY$ is any linear combination of the vector fields
$$
x \mapsto \cY^{ij}(x)=   x_i e_j- x_j e_i, \qquad 1 \le i < j \le N.
$$}
\end{remark}

In Section~\ref{sec:proof-gener-pohoz} we also deduce the following corollary from Theorem~\ref{generalized-pohozaev}.

\begin{corollary}
  \label{sec:ros-oton-generalized-inequality}
  Let $\cX \in C^{0,1}(\R^N,\R^N)$, and suppose that
  \begin{equation}
    \label{eq:c-1-c-2-condition}
\div\, \cX(x) \ge c_1 \qquad \text{and}\qquad \bigl(\cX(x)-\cX(y)\bigr) \cdot (x-y)   \leq c_2|x-y|^2 \qquad \text{for all $x,y \in \R^N$}
  \end{equation}
with constants $c_1,c_2 \in \R$. Moreover, let $u \in \cH^s_0(\Omega) \cap L^\infty(\Omega)$ be a (weak) solution of the problem~(\ref{eq:gradient-type-intro}) with a nonlinearity $f$ satisfying (\ref{eq:lipschitz-assumption}) and $F(t)= \int_0^t f(s)\,ds \ge 0$ for $t \in \R$. Then we have
\begin{equation}
  \label{eq:V-prime-formula-cor-2}
  \Gamma(1+s)^2 \int_{\partial \Omega} \psi_u^2\,  \cX \cdot \nu \,dx \le   \int_{\Omega} \Bigl(2c_2 N F(u)-\bigl[2 c_1-(N+2s) c_2 \bigr] f(u)u\Bigr)dx.  
  \end{equation}
  In particular, if $f(u)= |u|^{p-2}u$ for some $p>2$, then
\begin{equation}
  \label{eq:V-prime-formula-cor-2-2}
  \Gamma(1+s)^2 \int_{\partial \Omega} \psi_u^2\,  \cX \cdot \nu \,dx \le
  \Bigl(\frac{2c_2 N}{p}-\bigl[2 c_1-(N+2s) c_2 \bigr] \Bigr)  \int_{\Omega} |u|^pdx. 
  \end{equation}
\end{corollary}

Corollary~\ref{sec:ros-oton-generalized-inequality} gives rise to the nonexistence of nontrivial solutions of (\ref{eq:gradient-type-intro}) in the case where $u \mapsto f(u)= |u|^{p-2}u$ is a homogeneous nonlinearity with supercritical growth. In particular, the following non-existence result is an immediate consequence. 

\begin{corollary}
  \label{sec:supercritical-nonexistence-c-1-c-2}
  Let $\cX \in C^{0,1}(\R^N,\R^N)$ be a vector field satisfying
  \eqref{eq:c-1-c-2-condition} with some constants $c_2>0$ and $c_1 \in (\frac{c_2 N}{2},c_2N]$. Moreover, suppose that
  $$
  0< s < \bigl(\frac{c_1}{c_2}-\frac{N}{2}\bigr),\qquad 
  p> \frac{2N}{\frac{2c_1}{c_2}- (N+2s)},
  $$
  and let $u \in \cH^s_0(\Omega) \cap L^\infty(\Omega)$ be a (weak) solution of the problem
\begin{equation}
  \label{eq:gradient-type-intro-homogeneous}
(-\Delta)^s u =|u|^{p-2} u \quad \text{in $\Omega$},\qquad u = 0 \quad \text{on $\R^N \setminus \Omega$.}  
\end{equation}
If $\cX \cdot \nu \ge 0$ on $\partial \Omega$, then $u \equiv 0$.
\end{corollary}

If  (\ref{eq:c-condition}) holds for a vector field $\cX \in C^{0,1}(\R^N,\R^N)$  with some $c>0$, then, by Remark~\ref{remark-intro-1}, condition~(\ref{eq:c-1-c-2-condition}) holds with $c=c_2$ and $c_1 = Nc_2$. In this case, Corollary~\ref{sec:supercritical-nonexistence-c-1-c-2} reduces to the following statement. 
\begin{corollary}
  \label{sec:supercritical-nonexistence-c}
  Let $\cX \in C^{0,1}(\R^N,\R^N)$ be a vector field satisfying
  (\ref{eq:c-condition}) for some $c>0$.  Moreover, suppose that 
  $$
N \ge 2s \qquad \text{and}\qquad 
  p> \frac{2N}{N -2s},
  $$
  and let $u \in \cH^s_0(\Omega) \cap L^\infty(\Omega)$ be a (weak) solution of problem (\ref{eq:gradient-type-intro-homogeneous}). If $\cX \cdot \nu \ge 0$ on $\partial \Omega$, then $u \equiv 0$.
\end{corollary}

\begin{example}{\rm 
    We briefly discuss applications of Corollaries~\ref{sec:supercritical-nonexistence-c-1-c-2} and \ref{sec:supercritical-nonexistence-c} to some specific domains.\\[0.2cm]
(i) In the special case $\cX= \id$, Corollary~\ref{sec:supercritical-nonexistence-c} yields the nonexistence of nontrivial solutions of (\ref{eq:gradient-type-intro-homogeneous}) for starshaped domains, as stated in \cite[Corollary 1.3]{RX-Poh}.\\[0.2cm]
(ii) A specific example of a non-sharshaped domain $\Omega \subset \R^2$ to which Corollary~\ref{sec:supercritical-nonexistence-c} applies is given by
$$
\Omega = \{x \in \R^2\::\: x_1^2+ 10(x_2^3+x_1)^2 < 1\}
$$
Here, we choose the vector field
$$
\cX: \R^2 \to \R^2, \qquad \cX(x_1,x_2) = (5 x_1-4x_2, 5 x_2+4 x_1),
$$
so $\cX  \equiv 5 \id -4 \cY^{12}$ with the notation of Remark~\ref{remark-intro-1}. Hence (\ref{eq:c-condition}) is satisfied with $c=5$. Moreover, a careful estimate shows that $\cX \cdot \nu \ge 0$ on $\partial \Omega$ (see Figure 1).
\begin{figure}[h]
   \centering
\includegraphics[height=5cm]{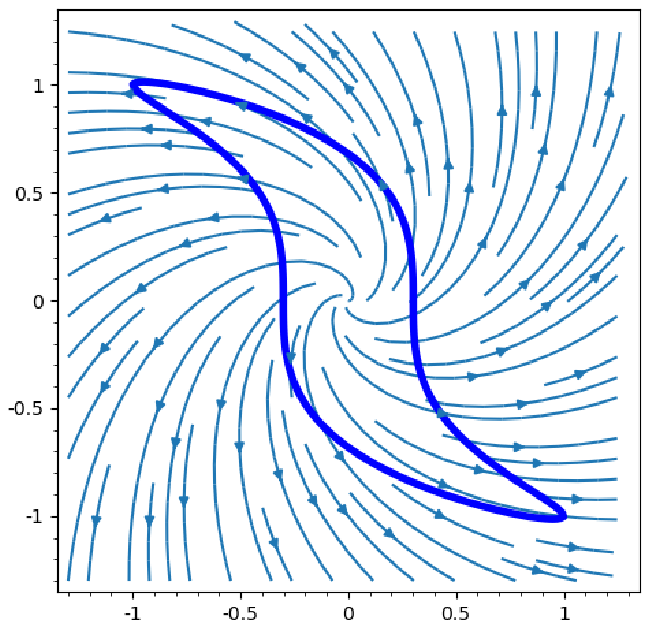}  
     	\caption{Domain $\Omega$ and flow lines of the vector field $\cX$.}
     \end{figure} 
In \cite[p. 92]{Reichel-book}, further (non-explicit) examples of non-sharshaped domains $\Omega \subset \R^N$ and vector fields $\cX$ of the form~(\ref{eq:c-cond-sufficient-con}) for some $c>0$ satisfying $\cX \cdot \nu \ge 0$ on $\partial \Omega$ are given in the context of the Dirichlet problem for the classical local equation $-\Delta u = |u|^{p-2}u$.\\[0.2cm]  
(iii) We consider $N \ge 2$ and, for $\eps \in [0,1)$, the vector field $\cX_\eps \in C^{0,1}(\R^N,\R^N)$ given by $\cX(x)= (\eps x_1,x_2,\dots,x_N)$, which satisfies $\div\, \cX_\eps \equiv  N-1+\eps$ and $\bigl(\cX(x)-\cX(y)\bigr) \cdot (x-y) \le |x-y|^2$
for $x,y \in \R^N$. Hence (\ref{eq:c-1-c-2-condition}) is satisfied with $c_2=1$ and
$c_1:= N-1+\eps \in [\frac{c_2 N}{2},c_2 N]$. Consequently, for any bounded domain $\Omega \subset \R^N$ satisfying
\begin{equation}
\label{eps-condition-outward-pointing}
\cX_\eps \cdot \nu \ge 0\qquad \text{on $\partial \Omega$,}  
\end{equation}
Corollary~\ref{sec:supercritical-nonexistence-c-1-c-2} yields nonexistence of nontrivial solutions to (\ref{eq:gradient-type-intro-homogeneous}) if $0< s < \min \{1, \frac{N}{2}-1 +\eps\}$ and $p> \frac{2N}{N-2(1+s-\eps)}$. To give specific examples, we restrict our attention to rotationally symmetric domains of the form 
      $$
      \Omega= \{x \in \R^N\::\: g(x_1^2) + \kappa \sum_{\ell=2}^Nx_\ell^2  < 0\}
      $$
      with $\kappa>0$ and a $C^{2}$-function $g: [0,\infty) \to \R$ having a simple zero at some point $r>0$ with the property that $g <0$ on $[0,r)$ and $g>0$ on $(r,\infty)$. Then $\Omega$ is a bounded domain of class $C^2$, and it can easily be shown that
      (\ref{eps-condition-outward-pointing}) holds for $\eps \ge 0$ sufficiently small, so in particular for $\eps = 0$. As an explicit example in dimension $N=2$, we consider the non-starshaped domain
      $$
      \Omega = \{x \in \R^2\::\: 3x_1^2-5x_1^4+x_1^6 -1 +4 y^4 < 0\}.
      $$
      In this case, a careful estimate shows that (\ref{eps-condition-outward-pointing}) holds with $\eps = \frac{1}{2}$ (see Figure 2 below). Hence Corollary~\ref{sec:supercritical-nonexistence-c-1-c-2} yields nonexistence of nontrivial solutions to (\ref{eq:gradient-type-intro-homogeneous}) for $s \in (0,\frac{1}{2})$ and $p> \frac{4}{1-2s}$.
\begin{figure}[h]
   \centering
\includegraphics[height=5cm]{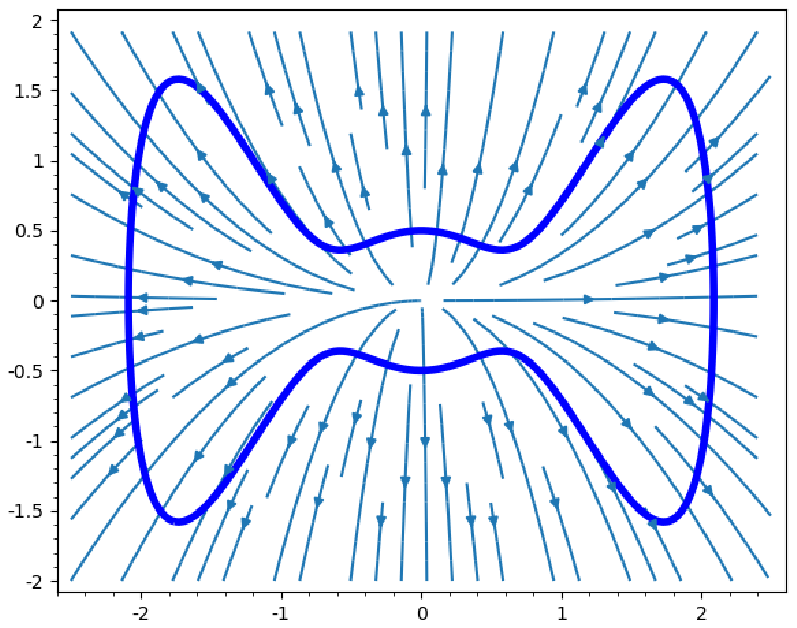}  
\caption{Domain $\Omega$ and flow lines of the vector field $\cX_{\text{\tiny $1/2$}}$.}
      \end{figure}
A related study of non-starshaped rotationally symmetric domains in the context of the second order semilinear elliptic PDEs is contained in \cite[Section 2]{mcgough}. 
}
\end{example}

Next, we briefly comment on the proof of Theorems~\ref{generalized-pohozaev} and \ref{Generalized-integration-by-parts}, which relies on some integral identities and boundary estimates obtained recently by the authors in \cite{SMT} to obtain a Hadamard formula for the rate of change of best constants in subcritical Sobolev embeddings with respect to domain deformations. In particular, this Hadamard formula applies to the first Dirichlet eigenvalue of $(-\Delta)^s$. It is one aim of the paper to indicate close connections between Hadamard formulas and Pohozaev identies in the fractional setting. In fact, both in the Hadamard formula given in \cite[Theorem 1.1]{SMT} and in fractional Pohozaev type identies, the boundary term on the LHS of (\ref{eq:V-prime-formula}) appears. Moreover, the bilinear form $\cE_{\cX}(u,v)$ related to the fractional deformation kernel $K_{\cX}$ defined in~(\ref{def-kernel-K-X}) arises as a derivative \linebreak $\frac{d}{d\eps}\big|_{\eps= 0}\, {\cE}(u \circ \Phi_\eps,v \circ \Phi_\eps)$, where ${\cE}$ is the unperturbed bilinear form given in (\ref{bilinear form}) and $\eps \to \Phi_\eps$ is a family of diffeomorphisms $\R^N \to \R^N$ with $\frac{d}{d\eps}\big|_{\eps= 0}\Phi_\eps=\cX$, see \cite{SMT} and Section~\ref{sec:from-doma-deform} below for more details. The connections will be further stressed in Theorem~\ref{smooth curves-thm} below, which is a variant of the Hadamard formula given in \cite[Corollary 1.2]{SMT} dealing with {\em arbitrary} Dirichlet eigenvalues of $(-\Delta)^s$. 

As a further application of the fractional Pohozaev identity in the form (\ref{pohozaev-rosoton-serra}), we derive, in Theorem~\ref{shape-derivative-lambda-k}, the simplicity of radial eigenvalues of $(-\Delta)^s$ in a ball or an annulus, and in Theorem~\ref{pohozaev-radial} we provide a multiplicity estimate in general (disconnected) radial open bounded sets. The proofs of these facts are extremely simple but have not been noticed in the literature up to our knowledge. As an application of Theorem~\ref{smooth curves-thm}, we also derive, in Theorem~\ref{shape-derivative-lambda-k}, a rate of change formula for radial eigenvalues with respect to radial deformations of balls or annuli.

The paper is organized as follows: Section \ref{sec:proof-gener-pohoz} is devoted to the proof of Theorem~\ref{Generalized-integration-by-parts}, Corollary~\ref{sec:ros-oton-generalized} and Corollary \ref{sec:ros-oton-generalized-inequality}. In the last section, we use the identity \eqref{eq:V-prime-formula} to derive Hadamard formula for simple eigenvalues of the Dirichlet fractional laplacian and we apply the latter to radial eigenvalues of bounded radial domains.

\section{Proof of the generalized integration by parts formula Theorem \ref{Generalized-integration-by-parts}}
\label{sec:proof-gener-pohoz}

This section is mainly devoted to the proof of Theorem~\ref{Generalized-integration-by-parts}. Throughout this section, let $\cX$ be a vector field of class $C^{0,1}$ which satisfies a global Lipschitz bound. Recall the definition in \eqref{eq-K-X-bilinear-form} of the bilinear form $\cE_{\cX}$ associated to $\cX$. We first need the following result.

\begin{lemma}\label{lem:j1kprimes-preliminary}
Let $U \in C^{\alpha}_c(\O)$ for some $\alpha>\max\{1,2s\}$. Then we have 
\begin{equation}
\label{first-kernel-formula}  
\cE_{\cX}(U,U)  = -2 \int_{\R^N} \nabla U \cdot \cX (-\Delta)^s U dx.
\end{equation}
\end{lemma}

A similar statement has been proved under slightly stronger regularity assumptions on $U$ in \cite[Lemma 4.2]{SMT}. Here we give a somewhat simpler proof which is also consistent with the present notation. 

\begin{proof}
By symmetry of the kernel and Fubini's theorem, we have 
\begin{align*}
  &\cE_{\cX}(U,U)\\
  &=               \frac{c_{N,s}}{2}\int_{\R^{2N}} (U(x)-U(y))^2 \Bigl[\frac{\div\, \cX(x) + \div\, \cX(y)}{|x-y|^{N+2s}}- (N+2s)\frac{ \bigl(\cX(x)-\cX(y)\bigr)\cdot (x-y)}{|x-y|^{N+2s+2}}\Bigr] dxdy\\
&= c_{N,s}\lim_{\mu \to 0} \int_{\R^N} \int_{\R^N \setminus \overline{B_\mu(y)}}(U(x)-U(y))^2 \Bigl[\frac{\textrm{div}\,\cX(x)}{|x-y|^{N+2s}} -(N+2s) \frac{(x-y)\cdot \cX(x)}{|x-y|^{N+2s+2}}dxdy \nonumber\\
  &= c_{N,s}\lim_{\mu \to 0} \int_{\R^N} \int_{\R^N \setminus \overline{B_\mu(y)}}(U(x)-U(y))^2 \,\nabla_x \Bigl(|x-y|^{-N-2s} \cX(x)\Bigr)\,dx dy. 
\end{align*}
Applying, for fixed $y \in \R^N$ and $\mu>0$, the divergence theorem in the domain $\R^N \setminus \overline{B_\mu(y)}$, we obtain 
\begin{align}
  &\int_{\R^{2N}} K_{\cX}(x,y)(U(x)-U(y))^2\,dxdy \label{poho-proof-1}\\
  &=c_{N,s}\lim_{\mu\to 0}     \int_{\R^N} \int_{\partial B_\mu(y)}  {(U(x)-U(y))^2} \frac{y-x}{|x-y|^{N+2s+1}}\cdot \cX(x) \, d\s(y)\,dx \nonumber\\
  &- 2c_{N,s}\lim_{\mu\to 0}
    \int_{\R^N} \int_{\R^N \setminus \overline{B_\mu(y)}} \frac{(U(x)-U(y))\n U(x) \cdot \cX(x) }{|x-y|^{N+2s}} dxdy\Bigr] = : I_1 - 2 I_2.\nonumber
\end{align}
Since $U \in C^{\alpha}_c(\Omega)$ for some $\alpha>2s$, we may use Fubini's theorem and the change of variable $(x,y) \mapsto (y,y-x)$ to see that
\begin{align}
  &I_2  =c_{N,s} \lim_{\mu\to 0}\int_{\R^N \setminus \overline{B_\mu(0)}} \int_{\R^N}  \frac{U(y)-U(y-x)}{|x|^{N+2s}}\n U(y) \cdot \cX(y)dy dx \nonumber\\
    &=\frac{c_{N,s}}{2} \int_{\R^{2N}}  \frac{2 U(y)-U(y-x)-U(y+x)}{|x|^{N+2s}}\n U(y) \cdot \cX(y)dy dx \nonumber\\
  &= \frac{c_{N,s}}{2} \int_{\R^{N}} \n U(y) \cdot \cX(y) \int_{\R^N} \frac{2 U(y)-U(y-x)-U(y+x)}{|x|^{N+2s}}dx dy \nonumber\\
 &=  \int_{\R^N}   \n U(y) \cdot \cX(y)(-\Delta)^sU(y) dy. \label{poho-proof-2}
\end{align}
Moreover, also by Fubini's theorem, we have
\begin{equation}
I_1 = \frac{1}{2} \lim_{\mu\to 0}\mu^{-N-1-2s}  \int_{|x-y|= \mu}  {(U(x)-U(y))^2} (y-x) \cdot (\cX(x)-\cX(y))    \, d\s(x,y) \label{poho-proof-3}
\end{equation}
Moreover, since $U$ is compactly supported, we may fix $R>0$ large enough such that
$(U(x)-U(y))^2 = 0$ for all $x,y \in B_R(0)$ with $|x-y|<1$.
Setting $N_\mu:= \{(x,y) \in B_R(0) \times B_R(0) \::\: |x-y|= \mu\}$ for $0<\mu<1$ and using that
$U, \cX \in C^{0,1}(\R^N)$, we thus deduce that 
\begin{align*}
  &\mu^{-N-1-2s} \int_{|x-y|= \mu}  {(U(x)-U(y))^2} (y-x) \cdot (\cX(x)-\cX(y))    \, d\s(x,y)\nonumber\\
  &= \mu^{-N-1-2s} \int_{N_\mu}  {(U(x)-U(y))^2} (y-x) \cdot (\cX(x)-\cX(y))\, d\s(x,y) = O(\mu^{3-1-2s}) \to 0,
\end{align*}
as $\mu \to 0$, since the $2N-1$-dimensional measure of the set $N_\mu$ is of order $O(N-1)$ as $\mu \to 0$. Thus \eqref{poho-proof-3} yields $I_1=0$, and together with (\ref{poho-proof-1}) and (\ref{poho-proof-2}) the claim follows.
\end{proof}
%

Next we consider $u \in \cH^s_0(\Omega)$ satisfying the assumptions of Theorem \ref{Generalized-integration-by-parts}, and we recall that
$u \in C_0^s(\overline \Omega) \cap C^\al_{loc}(\Omega)$ with $\al>\max(1,2s)$ by the standard regularity theory. We cannot apply Lemma~\ref{lem:j1kprimes-preliminary} directly to $u$ since $u$ does not have compact support. We therefore consider inner approximations $U_k:=u \zeta_k$ for $k \in \N$ for suitable functions $\z_k \in C^{1,1}(\R^N)$. To define $\z_k$, we note that, since $\Omega$ is of class $C^{1,1}$ by assumption, the signed distance function to $\partial \Omega$ is also of class $C^{1,1}$ in a neighborhood of $\de \O$. We therefore may consider a function $\tilde \delta \in C^{1,1}(\R^N)$ which is is positive in $\O$, negative in $\R^N\setminus \overline\O$ and coincides with the signed distance to $\partial \Omega$ in a neighborhood of $\partial \Omega$. We then define $\z_k \in C^{1,1}(\R^N)$ by 
$$
\z_k(x)= 1-\rho(k \tilde \delta(x)). 
$$
where $\rho\in C^\infty_c(-2,2)$ is fixed with $0 \le \rho \le 1$ and $\rho\equiv 1$ on $(-1,1)$. By \cite[Lemma 2.1 and 2.2]{SMT}, we have, for arbitrary $u \in \calH^s_0(\Omega)$,
\begin{equation}
  \label{eq:limit-prop-kernel}
u \zeta_k \to u \qquad \text{in $\calH^s_0(\Omega)$}\quad\text{and}\quad \cE_{\cX}(u\z_k,u\z_k)\to  \cE_{\cX}(u,u) \quad \text{as}\quad k \to \infty.
\end{equation}
Indeed, the latter is true for more general kernel functions $(x,y) \mapsto K(x,y)$ in place of $K_{\cX}$ as long as $(x,y) \mapsto K(x,y)|x-y|^{N+2s}$ defines a function in $L^\infty(\R^N \times \R^N)$. To complete the proof of Theorem~\ref{Generalized-integration-by-parts}, we need, as a final tool taken from\cite{SMT}, the following limit identity.

\begin{proposition}
  \label{prop-limit-identity}
  We have\footnote[1]{Note that sign of the RHS of (\ref{eq:limit-identity}) differs from \cite{SMT} since the inner unit normal is used in \cite{SMT}}
\begin{equation}
  \label{eq:limit-identity}
\lim_{k\to \infty}\int_{\Omega}\n (u \z_k) \cdot \cX \Bigl( u   \Ds \z_k  -I(u,\z_k)\Bigr)\,dx= \frac{\Gamma(1+s)^2}{2}  \int_{\de\O} \psi_u^2\, \cX\cdot \nu \, dx,
\end{equation}
where 
\be\label{eq:def-I-u-v-new-section}
I(u, \z_k)(x):= \frac{c_{N,s}}{2}\int_{\R^N}\frac{(u(x)-u(y))(\z_k(x)-\z_k(y))}{|x-y|^{N+2s}}\, dy.
\ee
\end{proposition}

\begin{proof}
As stated in \cite[Prop. 2.4 and Remark 2.5]{SMT}, the identity (\ref{eq:limit-identity}) holds for functions $u \in C_0^s(\overline \Omega) \cap C^{1}_{loc}(\Omega)$ satisfying, for some $\alpha>0$, the following regularity properties:
\be \label{eq:Bnd-reg-v-eps}
\psi_u \in C^{\a}(\ov\O) \qquad \text{and}\qquad 
 \delta^{1-\a} \n \psi_u \quad \text{is bounded in a neighborhood of $\partial \Omega$.}
\ee
The first property is satisfied under the assumption of Theorem~\ref{Generalized-integration-by-parts} by the regularity theory in \cite{RX}. Moreover, the gradient estimate for $\psi_u$ holds as well in this case, as proved in \cite{FS-2019}.

Hence the identity (\ref{eq:limit-identity}) holds if $u$ satisfies the assumptions of Theorem~\ref{Generalized-integration-by-parts}\footnote[2]{In \cite[Prop. 2.4]{SMT}, the property $u \z_k \in C^{1,1}_c(\Omega)$ for all $k$ is unnecessarily stated as an assumption. Only the bounds \ref{eq:Bnd-reg-v-eps} and the $C^{1,1}$-regularity of the functions $\z_k$ are used in the proof.} 
\end{proof}

We may now complete the
\begin{proof}[Proof of Theorem \ref{Generalized-integration-by-parts} ]
Applying Lemma~\ref{lem:j1kprimes-preliminary} to $U_k=u\z_k$, we find that 
$$
\cE_{\cX}(U_k,U_k)= -2 \int_{\R^N} \nabla U_k \cdot \cX (-\Delta)^s U_k dx= -2 \int_{\Omega} \nabla U_k \cdot \cX (-\Delta)^s U_k dx  \qquad \text{for $k \in \N$.}
$$
By the standard product rule for the fractional Laplacian, we have
$$
\Ds U_k=\z_k \Ds u + u\Ds \z_k -I(u, \z_k)\qquad \text{in $\Omega$}
$$
with $I(u, \z_k)$ given in \ref{eq:def-I-u-v-new-section}. 
We thus obtain
\begin{align}
\cE_{\cX}(U_k,U_k)&=-2 \int_{\O}   \z_k\n U_k \cdot \cX\Ds u\, dx -2 \int_{\R^N}\n U_k \cdot \cX \Bigl( u   \Ds \z_k  - I(u, \z_k)\Bigr)\,dx,\nonumber\\
&=-2 \int_{\O}   \z_k^2\n u \cdot \cX\Ds u\, dx -2 \int_{\R^N}\n U_k \cdot \cX \Bigl( u   \Ds \z_k  - I(u, \z_k)\Bigr)\,dx\nonumber\\
&-2\int_{\O}u\z_k\n\z_k\cdot\cX\Ds u\,dx.\nonumber
\end{align}
Since $\Ds u\in L^\infty(\O)$ and $u\in C^s_0(\ov\O)$, we easily find, by definition of $\z_k$, that 
\be\label{eqq}
\int_{\O}u\z_k\n\z_k\cdot\cX\Ds u\,dx\to 0\qquad\text{as}\qquad k\to\infty.
\ee
Taking the limit into the identity above and using \eqref{eq:limit-prop-kernel}, \eqref{eq:limit-identity} and \eqref{eqq} we deduce
\be \label{eq-limit-c-k}
\calE_{\cX}(u,u)=\lim_{k\to\infty}c_k(u)=-2\int_{\O}\n u\cdot\cX\Ds u\,dx-\G^2(1+s)\int_{\de\O} \psi_u^2\, \cX\cdot \nu \, dx.
\ee
The proof is finished.
\end{proof}

Next we give the

\begin{proof}[Proof of Corollary~\ref{sec:ros-oton-generalized-inequality}]
As noted in Remark~\ref{remark-intro-1}, it follows from assumption~(\ref{eq:c-condition}) that 
$$
\div \cX \equiv cN\qquad \text{and therefore}\qquad K_{\cX}(x,y) = \frac{c_{N,s}c}{2}(N-2s)|x-y|^{-N-2s}\quad \text{for $x,y \in \R^N$.}
$$
Hence for every weak solution $u \in \cH^s_0(\Omega) \cap L^\infty(\Omega)$ of (\ref{eq:gradient-type-intro}) we have 
$$
\calE_{\cX}(u,u)= c(N-2s){\cE}(u,u)= c (N-2s) \int_{\Omega}f(u)u\,dx. 
$$
Therefore (\ref{eq:V-prime-formula}) reduces to (\ref{eq:V-prime-formula-cor-1}) in this case, as claimed.
\end{proof}

We close this section with the

\begin{proof}[Proof of Corollary \ref{sec:ros-oton-generalized-inequality}]
  We first note that the second condition in (\ref{eq:c-1-c-2-condition}) implies that 
  \begin{equation}
    \label{eq:c-condition-variant-1}
\bigl(d\cX(y)h\bigr) \cdot h \le  c_2|h|^2 \qquad \text{for a.e. $y \in \R^N$ and every $h \in \R^N$.}
  \end{equation}
  Applying (\ref{eq:c-condition-variant-1}) to the coordinate vectors $e_1,\dots,e_N \in \R^N$ and combining the result with the first condition in (\ref{eq:c-1-c-2-condition}), we deduce that
  \begin{equation}
    \label{eq:two-sided-divergence-ineq}
  c_1 \le \div \,\cX \le  c_2 N \qquad \text{a.e. on $\R^N$.}
  \end{equation}
  In particular, this implies that
  $$
  K_{\cX}(x,y) \ge \frac{c_{N,s}}{2}\: \frac{2 c_1 - c_2(N+2s)}{|x-y|^{N+2s}}\qquad \text{for $x,y \in \R^N$.}
$$
Hence for every weak solution $u \in \cH^s_0(\Omega) \cap L^\infty(\Omega)$ of (\ref{eq:gradient-type-intro}) we have 
\begin{align*}
  \cE_{\cX}(u,u) &\ge \bigl[2 c_1 - c_2(N+2s) \bigr]{\cE}(u,u)= \bigl[2 c_1 - c_2(N+2s) \bigr] \int_{\Omega}f(u)u\,dx.
\end{align*}
Since $F(u)$ is nonnegative in $\Omega$ by assumption, we thus conclude from (\ref{eq:V-prime-formula}) and (\ref{eq:two-sided-divergence-ineq}) that 
\begin{align*}
  \Gamma(1+s)^2 \int_{\partial \Omega} \psi_u^2\,  \cX \cdot \nu \,dx   &=
                                                                          2\int_{\Omega} F(u) \div \, \cX\,dx -  \cE_{\cX}(u,u)\\
&\le \int_{\Omega} \Bigl(2c_2N F(u)-\bigl[ 2 c_1 - c_2(N+2s) \bigr] f(u)u\Bigr)dx,  
\end{align*}
as claimed in (\ref{eq:V-prime-formula-cor-2}). Moreover, (\ref{eq:V-prime-formula-cor-2-2}) is a direct consequence of (\ref{eq:V-prime-formula-cor-2}).
\end{proof}

\section{A Hadamard formula for the fractional Dirichlet eigenvalue problem}
\label{sec:from-doma-deform}

Let $\Omega \subset \R^N$ denote a bounded open set with $C^{1,1}$-boundary.
From now on we fix $\eps_0>0$ and a family of deformations $\{\Phi_\eps\}_{\eps \in (-\eps_0,\eps_0)}$ with the following properties:
\be\label{eq:def-diffeom}
\begin{aligned}
&\text{$\Phi_\eps \in C^{1,1}(\R^N; \R^N)$ for $\eps  \in (-1,1)$, $\Phi_0= \id_{\R^N}$, and}\\
&\text{the map $(-1,1) \to C^{1,1}(\R^N,\R^N)$, $\eps \to \Phi_\eps$ is of class $C^2$.}  
 \end{aligned}
 \ee
By making $\eps_0>0$ smaller if necessary, we may then assume that $\Phi_\e:\R^N\to \R^N$ is a global diffeomorphism for $\eps \in (-\eps_0,\eps_0)$, see e.g. \cite[Chapter 4.1]{delfour-zolesio}.
For $\e \in (-\eps_0,\eps_0)$, we write $\Omega_\eps = \Phi_\eps(\Omega)$.

The aim of this section is to establish the following rate of change formula for Dirichlet eigenvalues of the fractional Laplacian with respect to the domain deformation given by $\Phi_\eps$.

\begin{theorem}
  \label{smooth curves-thm}
  Consider a $C^1$-curve
$$
  (-\eps_0,\eps_0) \to \cH^s_0(\Omega), \qquad \eps \mapsto u_\eps
$$
with the property that, for every $\eps \in (-\eps_0,\eps_0)$, the function
$$
v_\eps:= u_\eps \circ \Phi_\eps^{-1} \quad \in \cH^s_0(\Omega_\eps)
$$
is a nontrivial weak solution of the Dirichlet eigenvalue problem
\begin{equation}
  \label{eq:perturbed-eigenvalue}
(-\Delta)^s v_\eps = \lambda(\eps) v_\eps \quad \text{in $\O_\eps$,}\qquad v_\eps \equiv 0 \quad \text{in $\O_\eps^c$},
\end{equation}
for some $\lambda(\eps) \in \R$. Then the function $\eps \mapsto \lambda_\eps$ is of class $C^1$ on $(-\eps_0,\eps_0)$, and we have the identity
  $$
  \partial_\eps \Big|_{\eps=0} \lambda(\eps) = -   \frac{\Gamma(1+s)^2 \int_{\partial \Omega} \psi_u^2\,  \cX \cdot \nu \,dx}{\int_{\Omega}u^2\,dx}
  $$
  with $u:= u_0$ and the vector field
  \begin{equation}
    \label{eq:def-X-rate-of-change}
  \cX:= \partial_{\eps} \Big|_{\eps = 0}\Phi_\eps \in C^{1,1}(\R^N,\R^N).  
  \end{equation}
\end{theorem}

For the proof of this theorem, we introduce some notation. In weak sense, the eigenvalue equation for $v_\eps$ reads
\begin{equation}
  \label{eq:perturbed-eigenvalue-weak}
{\cE}(v_\eps,\phi) = \lambda_\eps \int_{\Omega_\eps} v_\eps \phi\,dx \qquad \text{for all $\phi \in \cH^s_0(\Omega_\eps)$,}
\end{equation}
Using the fact that the map
  $$
\cH^s_0(\Omega) \to \cH^s_0(\Omega_\eps),\qquad \psi \mapsto \psi \circ \Phi_\eps^{-1}
$$
is a topological isomorphism, we may rewrite this property, by means of integral transformations, in the form
\begin{equation}
  \label{eq:perturbed-eigenvalue-weak-fixed-domain}
{\cE}^\eps(u_\eps,\phi) = {\l(\e)} \int_{\Omega} u_\eps \phi \textrm{Jac}_{\Phi_\e} \,dx \qquad \text{for all $\phi \in \cH^s_0(\Omega)$.}
\end{equation}
Here $\textrm{Jac}_{\Phi_\e}$ denotes the Jacobian determinant of the map
$\Phi_\eps \in C^{1,1}(\R^N,\R^N)$, and 
\be\label{eq:def-j-Uk}
{\cE}^\eps(v,\phi):= \frac{1}{2}\int_{\R^{2N}}  {(v(x)-v(y))(\phi(x)-\phi(y))}  K_\e(x,y)    dxdy  \qquad \text{for $v,\phi \in \cH^s_0(\Omega)$} 
\ee
with the kernel 
\be \label{eq:def-K-eps}
K_\e(x,y):=  c_{N,s}\frac{\textrm{Jac}_{\Phi_\e}(x)\textrm{Jac}_{\Phi_\e}(y)}{|\Phi_\e(x)-\Phi_\e(y)|^{N+2s}}\qquad\textrm{for $\eps \in (-\eps_0,\eps_0)$}.
\ee
The first step in the proof of Theorem~\ref{smooth curves-thm} is the following lemma.

\begin{lemma}
  \label{diff-Riesz-1}
  \begin{itemize}
  \item[(i)] The map 
$$
(-\eps_0,\eps_0) \times \cH^s_0(\Omega) \times \cH^s_0(\Omega) \to \R, \qquad
(\eps,v,w) \mapsto {\cE}^\eps(v,w)
$$
is of class $C^1$ with
\begin{align}
  \widetilde {\cE}(v,w):= \partial_\eps \Big|_{\eps = 0}{\cE}^\eps(v,w) = 
 \cE_{\cX}(u,v) \label{cV-v-diff-property-zero}
\end{align}
for $v,w \in \cH^s_0(\Omega)$, where $\cX$ is given in (\ref{eq:def-X-rate-of-change}) and the kernel $\cE_{\cX}(v,w)$ is defined in \eqref{eq-K-X-bilinear-form}.  
\item[(ii)] The map
  $$
(-\eps_0,\eps_0) \times L^2(\Omega) \times L^2(\Omega) \to \R, \qquad
(\eps,v,w) \mapsto \int_{\Omega} v w \textrm{Jac}_{\Phi_\e} \,dx
$$
is of class $C^1$ with
$$
\de_\e\Big|_{\e=0}  \int_{\Omega} v w \textrm{Jac}_{\Phi_\e} \,dx =
\int_{\Omega} v w \textrm{div}\,\cX\,dx \qquad \text{for $v,w \in L^2(\Omega)$.}
$$
\end{itemize}
\end{lemma}

\begin{proof}
We only give the proof of (i), the proof of (ii) is similar but easier. We first note that, by direct computation, we have 
\begin{align}
  &\de_\e\Big|_{\e=0}\Bigl(|x-y|^{N+2s}  K_\e(x,y)\Bigr) \label{eq:estim-K-eps-2}\\
  &=  c_{N,s} \left\lbrace (\textrm{div}\cX(x)+ \textrm{div}\cX(y)) - (N+2s)\frac{(\cX(x)-\cX(y))(x-y)}{|x-y|^2}-\right\rbrace = 2 |x-y|^{N+2s} K_\cX(x,y) \nonumber
\end{align}
uniformly in $x,y \in \R^N$. Next we consider the space $\cL^2_S(\cH^s_0(\Omega))$ of continuous symmetric bilinear forms on $\cH^s_0(\Omega)$, which is endowed with the norm
  $$
  \|b\|_{\cL^2_s}:= \sup  \Bigl\{\frac{|b(v,w)|}{\|v\|_{H^s}\|w\|_{H^s}} \::\: v,w \in \cH^s_0(\Omega) \setminus \{0\} \Bigr\}
  $$
  It then suffices to show that
  \begin{equation}
    \label{eq:sufficient-bilinear-diff}
  \text{the map $(-\eps_0,\eps_0) \to \cL^2_S(\cH^s_0(\Omega))$, $\eps \mapsto {\cE}^\eps$ is of class $C^1$ with $\partial_\eps \big|_{\eps=0}\, {\cE}^\eps =  \widetilde {\cE},$}   
\end{equation}
where $\widetilde {\cE}$ is defined in \eqref{cV-v-diff-property-zero}.
To see the differentiability at $\eps=0$, we note that, since 
$$
K_\e(x,y)-c_{N,s}|x-y|^{-N-2s} - 2K_\cX(x,y) = o(\eps)|x-y|^{-N-2s}
$$
by (\ref{eq:estim-K-eps-2}), we have, 
  for $v,w \in \cH^1_0(\Omega)$,
  \begin{align*}
    {\cE}^\eps&(v,w) -{\cE}(v,w)- \widetilde{{\cE}}(v,w)\\
    &= \frac{1}{2}\int_{\R^{2N}}  {(v(x)-v(y))(w(x)-w(y))}  \bigl(K_\e(x,y)-c_{N,s}|x-y|^{-N-2s} - 2K_\cX \bigr)dxdy\\
 &=o(\eps)\int_{\R^{2N}}  {(v(x)-v(y))(w(x)-w(y))}K_0(x,y)dxdy \le o(\eps) \|v\|_{H^1_0}\|w\|_{H^1_0}
  \end{align*}
  as $\eps \to 0$, where $o(\eps)$ is independent of $v$ and $w$. Consequently, 
  $$
  \|{\cE}^\eps- {\cE} - \eps \widetilde {\cE}\|_{\cL^2_S} =
  \sup \Bigl\{\frac{|{\cE}^\eps(v,w)-{\cE}(v,w)- \widetilde{{\cE}}(v,w)|}{\|v\|_{H^1_0}\|w\|_{H^1_0}} \::\: v,w \in \cH^1_0(\Omega) \setminus \{0\} \Bigr\}= o(\eps),
  $$
  and this shows that the map $\eps \mapsto {\cE}^\eps$ is differentiable at $\eps = 0$ with $\partial_\eps \big|_{\eps=0}\,{\cE}^\eps =  \widetilde {\cE}$.  For $\eps_* \in (-\eps_0,\eps_0)$ different from $0$, the same argument shows that $\eps \mapsto {\cE}^\eps$ is differentiabe at $\eps_*$, where $\partial_{\eps} \big|_{\eps = \eps_*}\, {\cE}^\eps$ has the same form as $\partial_{\eps} \big|_{\eps = 0}\, {\cE}^\eps$ in (\ref{cV-v-diff-property-zero}) with $\cX$ replaced by $\cX_{\eps_*}:= \partial_{\eps} \big|_{\eps = \eps_*}\,\Phi_\eps$. Finally, the continuity of the map
  $$
  (-\eps_0,\eps_0) \to \cL^2_S(\cH^s_0(\Omega)), \qquad \eps_* \mapsto \partial_{\eps} \big|_{\eps = \eps_*}\, {\cE}^\eps,
  $$
follows in a straightforward way from the fact that the map  
  $$
  (-\eps_0,\eps_0) \to L^\infty(\R^N \times \R^N),\qquad
  \eps \mapsto |x-y|^{N+2s}K_{\cX_\eps}(x,y)
  $$
is continuous. The proof of (\ref{eq:sufficient-bilinear-diff}) is thus finished.   
\end{proof}

We may then complete the 

\begin{proof}[Proof of Theorem~\ref{smooth curves-thm}]
Since ${\l(\e)}= \frac{{\cE}^\eps(u_\eps,u_\eps)}{\int_{\Omega} u_\eps^2 \textrm{Jac}_{\Phi_\e}\,dx}$ for $\eps \in (-\eps_0,\eps_0)$, it follows from Lemma~\ref{diff-Riesz-1} that the map $\eps \mapsto {\l(\e)}$ is of class $C^1$. So we may differentiate the equation
  $$
  {\cE}^\eps(u_\eps,u_\eps) = {\l(\e)} \int_{\Omega} u_\eps^2 \textrm{Jac}_{\Phi_\e}\,dx,
  $$
at $\eps = 0$, noting that
  $$
  \partial_\eps \Big|_{\eps=0}    {\cE}^\eps(u_\eps,u_\eps) = \widetilde {\cE}(u,u) + 2 {\cE}(\partial_\eps \big|_{\eps=0}\, u_\eps,u)
  $$
  and
  $$
  \partial_\eps \Big|_{\eps=0} \Bigl({\l(\e)} \int_{\Omega} u_\eps^2 \textrm{Jac}_{\Phi_\e}\,dx \Bigr) = \Bigl(  \partial_\eps \big|_{\eps=0}\, {\l(\e)}\Bigr) \int_{\Omega} u^2 \,dx + 2 \lambda(0) \int_{\Omega} \Bigl(\partial_\eps \big|_{\eps=0}\, v_\eps\Bigr) \phi \,dx + \lambda(0) \int_{\Omega} u^2 \div\,\cX\,dx.
  $$
 Since
  $$
  {\cE}(\partial_\eps \big|_{\eps=0} u_\eps,u)= \lambda(0) \int_{\Omega} \Bigl(\partial_\eps \big|_{\eps=0} u_\eps\Bigr)u \,dx
  $$
as a consequence of (\ref{eq:perturbed-eigenvalue-weak-fixed-domain}), we deduce that
  $$
  \widetilde {\cE}(u,u) = \Bigl(  \partial_\eps \big|_{\eps=0}\, {\l(\e)}\Bigr) \int_{\Omega} u^2 \,dx + \lambda(0) \int_{\Omega} u^2 \div\,\cX\,dx.
  $$
On the other hand, the generalized Pohozaev identity (\ref{eq:V-prime-formula}) for $u$ gives, with $F(u)=\frac{\lambda_0}{2}u^2$,
  $$
  \int_{\R^{2N}} K_{\cX}(x,y)(u(x)-u(y))^2\,dxdy =
  \lambda_0\int_{\Omega}u^2 \div \, \cX\,dx -   \Gamma(1+s)^2 \int_{\partial \Omega} \psi_u^2\,  \cX \cdot \nu \,dx 
$$
Recalling (\ref{cV-v-diff-property-zero}), we conclude that
$$
\partial_\eps \big|_{\eps=0}\, {\l(\e)}  = - \frac{ \Gamma(1+s)^2 \int_{\partial \Omega} \psi_u^2\,  \cX \cdot \nu \,dx}{\int_{\Omega} u^2 \,dx},
$$
as claimed.
\end{proof}

\section{Application to the radial eigenvalue problem}

In this section we study the eigenvalue problem
\be
\label{eq-eigenvalue-problem}
(-\Delta)^s w = \mu w \quad \text{in $\O$,}\qquad u \equiv 0 \quad \text{in $\O^c$},
\ee
among radial functions. For this, we let $H_{rad}^s$ denote the subspace of radially symmetric functions in the space $\calH^s_0(\O)$. 
By definition, a function $w \in H_{rad}^s$ is an eigenfunction of \eqref{eq-eigenvalue-problem} corresponding to the eigenvalue $\mu$ if 
\be\label{eq-eigenvalue-problem-weak-sense}
\calE_s(w,\psi)=\mu(\O)\int_{\O}w(x)\psi(x)dx \qquad \text{for all $\psi \in H_{rad}^s$.}
\ee

In the following, we will call $\mu$ a {\em radial eigenvalue} for $\mu$ if there exists an eigenfunction $w \in H_{rad}^s$ for $\mu$. 
It is a well-known fact that the radial eigenvalues of \eqref{eq-eigenvalue-problem} form an increasing sequence  of numbers $0<\mu_1<\mu_2\leq\mu_3\leq\cdots\nearrow+\infty$, counted with possible multiplicity.

While the simplicity of $\mu_1$ is a classical fact (see e.g \cite{FP}), the same property seems unavailable in the literature for higher eigenvalues. In this section, we shall show, by means of the fractional
Pohazaev identity (\ref{pohozaev-rosoton-serra}), that all radial eigenvalues are simple in the case where $\Omega$ is a ball or an annulus in $\R^N$.

For a related question, we refer to \cite{RF} where for $\O=\R^N$ simplicity result has been obtained for Schr\"odinger operator with a increasing radially symmetric potential.
The second aim of this section is to derive, from Theorem~\ref{smooth curves-thm}, a Hadamard formula for the dependence of the $k$-th eigenvalue $\mu_k$ on the inner and outer radius of $\Omega$.

The following is the main result of this section. Here and in the following, we identify a radial function $u=u(x)$ with the associated function $u=u(r)$ of the radial variable.

\begin{theorem}\label{shape-derivative-lambda-k}
Let $0< r_{inn}< r_{out}< \infty$ and suppose that either 
\begin{equation*}
\Omega = B_{r_{out}}(0) \qquad \text{or}\qquad \Omega = A(r_{inn},r_{out}):=\{x \in \R^N\::\: r_{inn} < |x| < r_{out}\}.
\end{equation*}
  Let $k\geq 1$ and let $\l_k$ be the $k$-th radial eigenvalue of \eqref{eq-eigenvalue-problem}. Then we have:
  \begin{itemize}
\item[(i)] $\l_k(\O)$ is simple.   
\item[(ii)] $\lambda_k$ depends in a differentiable way on $r_{out}$ with 
\be
  \label{der-th-k-out}
\frac{\partial_{\lambda_k}}{\partial r_{out}} = -\G(1+s)^2 |S^{N-1}| r_{out}^{N-1}
\psi_u^2(r_{out}).
\ee
Moreover, in the case where $\Omega = A(r_{inn},r_{out})$, $\l_k$ depends in a differentiable way on $r_{inn}$ with 
\be
  \label{der-th-k-inn}
\frac{\partial_{\lambda_k}}{\partial r_{inn}} = \G(1+s)^2 |S^{N-1}| r_{inn}^{N-1}
\psi_u^2(r_{inn})
\ee
Here $u \in H^s_{rad}$ is the (up to sign unique) $L^2$-normalized eigenfunction associated with $\lambda_k$, and
$\psi_u$ is the continuous extension of $\frac{u}{\delta^s}$ to $\overline \Omega$, as before. 
\end{itemize}
\end{theorem}

Notice that the statement of the theorem is new for $k>1$ but is already known for $k=1$. In fact, as we already mentioned above, the simplicity of the first (radial) eigenvalue is a classical fact, while the identities \eqref{der-th-k-out} and \eqref{der-th-k-inn} follow from \cite[Corollary 1.2]{SMT} in this special case.

In the case $N=1$ the annulus $A(r_{inn},r_{out})$ is a disconnected set. It is therefore natural to ask whether at least a weaker variant of Theorem~\ref{shape-derivative-lambda-k} still holds on other disconnected radial open sets. The following result gives a partial answer to this question.

\begin{theorem}
\label{higher-mult}  
Consider, for some $n \in \N$, real positive numbers
$$
0< r_{inn}^1< r_{out}^1< r_{inn}^2 < r_{out}^2 < \dots < r_{inn}^n < r_{out}^n<\infty
$$
and suppose that either 
\begin{equation*}
  \Omega = B_{r_{out}^1}(0)\cup \bigcup_{i=2}^n A(r_{inn}^i, r_{out}^i) \quad \text{or}\quad
  \Omega = \bigcup_{i=1}^n A(r_{inn}^i, r_{out}^i).
\end{equation*}
Then every radial eigenvalue of \ref{eq-eigenvalue-problem} on $\Omega$ has multiplicity at most $n$.
\end{theorem}
 
Note that Theorem~\ref{shape-derivative-lambda-k}(i) is a special case of Theorem~\ref{higher-mult}. In the following, we therefore give the proof of Theorem~\ref{higher-mult} first and then add the proof of Theorem~\ref{shape-derivative-lambda-k}(ii).

We start by noting the following direct consequence of the fractional Pohozaev type identity (\ref{pohozaev-rosoton-serra}).

\begin{proposition}
  \label{pohozaev-radial}
Consider, for some $n \in \N$, real positive numbers
$$
0< r_{inn}^1< r_{out}^1< r_{inn}^2 < r_{out}^2 < \dots < r_{inn}^n < r_{out}^n<\infty
$$
and suppose that either 
$$  \Omega = B_{r_{out}^1}(0)\cup \bigcup_{i=2}^n A(r_{inn}^i, r_{out}^i) \quad \text{or}\quad
  \Omega = \bigcup_{i=1}^n A(r_{inn}^i, r_{out}^i).
$$
Moreover, let $u \in H_{rad}^s$ be a solution of \eqref{eq-eigenvalue-problem}. 
 \begin{itemize}
\item[(i)] We have
\begin{equation}
\label{equiv-prop}  
  u = 0 \qquad \text{if and only if} \qquad \psi_u(r_{out}^1)=\psi_u(r_{out}^2)= \dots = \psi_u(r_{out}^n)=0. 
\end{equation}
\item[(ii)]  If $\Omega = \bigcup \limits_{i=1}^n A(r_{inn}^i, r_{out}^i)$, then 
  \begin{equation}
    \label{eq:annulus-formula}
    \int_{\Omega}u^2\,dx=\frac{|S^{N-1}| \Gamma(1+s)^2}{2s \mu}\sum_{i=i}^n\Bigl(\bigl(r_{out}^i\bigr)^N \psi_u^2(r_{out}^i) - \bigl(r_{inn}^i\bigr)^N \psi_u^2(r_{inn}^i)\Bigr).
  \end{equation}
\item[(iii)] If $\Omega = B_{r_{out}^1}(0)\cup \bigcup \limits_{i=2}^n A(r_{inn}^i, r_{out}^i)$, then
  \begin{align}
    \label{eq:ball-formula}
    \int_{\Omega}u^2\,dx &= \\
    &\frac{|S^{N-1}| \Gamma(1+s)^2}{2s \mu}\Bigl(\bigl(r_{out}^1\bigr)^N \psi_u^2(r_{out}^1) + \sum_{i=2}^n\bigl(r_{out}^i\bigr)^N \psi_u^2(r_{out}^i) - \bigl(r_{inn}^i\bigr)^N \psi_u^2(r_{inn}^i)\Bigr).\nonumber
  \end{align}
\end{itemize}
\end{proposition}

Here, as noted before, we write $u$ and $\psi_u$ as a function of the radial variable.

\begin{proof}
Applying (\ref{pohozaev-rosoton-serra}) with $f(u)= \mu u$ and $F(u) = \frac{\mu}{2}u^2$, we obtain 
\begin{equation}
  \label{eq:pohozaev-specialized}
\int_{\Omega}u^2\,dx = \frac{\Gamma(1+s)^2}{2s \mu}\int_{\partial \Omega} \Bigl(\frac{u}{\delta^s}\Bigr)^2 x \cdot \nu\, d\sigma(x),
\end{equation}
where, as before, $\nu$ denotes the outer unit normal on $\partial \Omega$.
The formulas (\ref{eq:ball-formula}) and (\ref{eq:annulus-formula}) follow directly from (\ref{eq:pohozaev-specialized}) and the radiality of $u$ in view of the fact that the outward unit normal
  $\nu$ on $\partial \Omega$ is given by
  $$
  \nu(x)= \left\{
  \begin{aligned}
    &\frac{x}{|x|} &&\qquad \text{if $|x|= r_{out}^i$ for some $i \in \{1,\dots,N\}$;}\\
    &-\frac{x}{|x|} &&\qquad \text{if $|x|= r_{inn}^i$ for some $i \in \{1,\dots,N\}$.}
  \end{aligned}
\right.
$$
To see (\ref{equiv-prop}), we note that $u = 0$ trivially implies $\psi_u(r_{out}^i)=0$ for $i=1,\dots,n$. On the other hand, if $\psi_u(r_{out}^i)=0$ for $i = 1,\dots,n$, it follows from (\ref{eq:ball-formula}) and (\ref{eq:annulus-formula}) that $\int_{\Omega}u^2\,dx \le 0$ and therefore $u =0$. Hence (\ref{equiv-prop}) follows.
\end{proof}

We may now complete the

\begin{proof}[Proof of Theorem~\ref{higher-mult}]
  Let, for $\mu>0$, $V_\mu \subset H^s_{rad}$ denote the space of radial solutions of the eigenvalue problem \eqref{eq-eigenvalue-problem}. From (\ref{equiv-prop}), it follows that the linear map
  $$
  \ell: V_\mu \to \R^n, \qquad \ell(u)= \Bigl(\psi_u(r_{out}^1),\dots,\psi_u(r_{out}^n)\Bigr)
  $$
  is injective. Hence the space $V_\mu$ is at most $n$-dimensional. It thus follows that every positive eigenvalue $\mu$ of \eqref{eq-eigenvalue-problem} has multiplicity at most $n$.
  \end{proof}

  In the remainder of this section, we give the

  \begin{proof}[Proof of Theorem~\ref{shape-derivative-lambda-k}(ii)]
 We only prove the differentiability of $\lambda_k$ as a function of $r_{out}$ and the formula \eqref{der-th-k-out}, the differentiability as a function of $r_{inn}$ and the formula \eqref{der-th-k-inn} follow in a similar way.

 We fix $\delta>0$ with $\delta < r_{out}$ in case
 $\Omega = B_{r_{out}}(0)$ and $\delta < r_{out}-r_{inn}$ in case $\Omega = \Omega = A(r_{inn},r_{out})$. Moreover, we let 
  $\cX \in C^{1,1}(\R^N,\R^N)$ be a vector field with
  $$
  \cX(x)= \frac{x}{|x|} \qquad \text{for $x \in A(r_{out}-\frac{\delta}{2},r_{out}+\frac{\delta}{2})$}
    $$
    and
    $$
    \cX \equiv 0 \qquad \text{in $\R^N \setminus A(r_{out}-\delta,r_{out}+\delta)$.}
$$
  For $\eps \in \R$, we now define $\Phi_\eps \in C^{1,1}(\R^N,\R^N)$, $\Phi_\eps(x)= x + \eps \cX(x)$. Then $\Phi_\eps$ satisfies the assumptions \eqref{eq:def-diffeom}, so we may fix $\eps_0 \in (0,\frac{\delta}{2})$ sufficiently small so that $\Phi_\e:\R^N\to \R^N$ is a global diffeomorphism for $\eps \in (-\eps_0,\eps_0)$.
  Moreover, for $\e \in (-\eps_0,\eps_0)$, we write
  $\Omega_\eps = \Phi_\eps(\Omega)$. By our choice of $\delta$ we have
  \begin{equation}
    \label{eq:omega-eps-case-1}
  \Omega_\eps = B_{r_{out}+\eps}(0) \qquad \text{if}\quad \Omega = B_{r_{out}}(0)
  \end{equation}
  and
  \begin{equation}
    \label{eq:omega-eps-case-2}
  \Omega_\eps = A(r_{inn},r_{out}+\eps) \qquad \text{if}\quad \Omega = A(r_{inn},r_{out}).
  \end{equation}
  Next, for $\eps \in (-\eps_0,\eps_0)$, we let ${\l(\e)}$ denote the $k$-th eigenvalue of (\ref{eq-eigenvalue-problem}) on $\Omega = \Omega_\eps$. By Proposition~\ref{pohozaev-radial}(i), there exists a unique eigenfunction $v_\eps \in \cH^s_0(\Omega_\eps)$ corresponding to ${\l(\e)}$ with  $\psi_{v_\eps}(r_{out}+\eps)>0$ and the normalization 
  $\|u_\eps\|_{L^2(\Omega)}=1$, where we define 
  $$
  u_\eps:= v_\eps \circ \Phi_\eps \qquad \text{for $\eps \in (-\eps_0,\eps_0)$.}
  $$
  We claim that 
  \begin{equation}
    \label{eq:sufficent-differentiability}
    \text{the curve $(-\eps,\eps) \to \cH^s_0(\Omega), \quad \eps  \mapsto u_\eps$ is of class $C^1$.}
  \end{equation}
  Once this is proved, it follows from Theorem~\ref{smooth curves-thm} and the definition of $\cX$ that $\eps \mapsto {\l(\e)}$ is a differentiable function with
  $$
\partial_\eps \big|_{\eps = 0}{\l(\e)} = - \Gamma(1+s)^2 \int_{\partial \Omega} \psi_u^2\,  \cX \cdot \nu \,dx =  -\G(1+s)^2 |S^{N-1}| r_{out}^{N-1}
\psi_u^2(r_{out}).
  $$
  Thus \eqref{der-th-k-out} follows by (\ref{eq:omega-eps-case-1}) and (\ref{eq:omega-eps-case-2}). As mentioned before, (\ref{der-th-k-inn}) follows by a similar argument.

  It thus remains to prove (\ref{eq:sufficent-differentiability}). More precisely, it suffices to prove, using the simplicity of the eigenvalue ${\l(\e)}$ and the implicit function theorem, the differentiability of the map $\eps \to u_\eps$ in a neighborhood of $\eps = 0$. 
For $\eps \in (-\eps_0,\eps_0)$, we define the linear maps
$$
L_\eps  \in \cL \bigl(H^s_{rad}, (H^s_{rad})'\bigr), \qquad [L_\eps v]w= {\cE}^\eps(v,w)
$$
and
$$
J_\eps \in \cL \bigl(H^s_{rad}, (H^s_{rad})'\bigr), \qquad [J_\eps v]w= \int_{\Omega}v w \textrm{Jac}_{\Phi_\e}\,dx.
$$
Here, as usual, $(H^s_{rad})'$ denotes the topological dual of $H^s_{rad}$. With this notation, we can write the property (\ref{eq:perturbed-eigenvalue-weak-fixed-domain}) in the form
\begin{equation}
  \label{eq:eigenvalue-eps-abstract-form}
L_\eps u = \lambda J_\eps u \qquad \text{in $(H^s_{rad})'$.}
\end{equation}

Moreover, as a consequence of Lemma~\ref{diff-Riesz-1}, we see that the maps 
    $$
  (-\eps_0,\eps_0) \to  \cL\bigl(H^s_{rad},\bigl(H^s_{rad}\bigr)'\bigr),\qquad  \eps \mapsto L_\eps, \qquad \eps \mapsto J_\eps
  $$
  are of class $C^1$. Consequently, the map 
\begin{equation}
  \label{eq:def-F}
  \Sigma:(-\eps_0,\eps_0) \times (0,\infty) \times H^s_{rad} \to  \R \times (H^1_{rad})', \qquad \Sigma(\eps,\lambda,u)=(\|u\|_{L^2(\Omega)}^2-1, L_\eps u- \lambda  J_\eps u)
\end{equation}
is also of class $C^1$, and by definition we have
\begin{equation}
  \label{eq:eps-solution}
\Sigma(\eps,{\l(\e)},u_\eps)= 0 \qquad \text{for $\eps \in (-\eps_0,\eps_0)$.}
\end{equation}
Moreover, we have
\begin{equation*}
\frac{\partial \Sigma}{\partial (\lambda,u)}(0,\lambda,u)(\mu,v) = \Bigl(2 \langle u,v \rangle_{L^2(\Omega)}, L_0 v -\lambda v - \mu J_0(u \Bigr)
\end{equation*}
for $(\lambda,u) \in (0,\infty) \times H^s_{rad}$ and $(\mu,v) \in \R \times H^s_{rad}$. We claim that 
\begin{equation}
  \label{eq:claim-isomorphism}
\frac{\partial \Sigma}{\partial (\lambda,u)}(0,\lambda(0),u_0) \in \cL\Bigl(\R \times H^1_{rad},\R \times \bigl(H^1_{rad}\bigr)'\Bigr)\qquad \text{is a topological isomorphism.}     
\end{equation}
Indeed, since the radial eigenvalue $\lambda(0)$ is simple by Theorem~\ref{shape-derivative-lambda-k}(i), the linear map
  $$
  v \mapsto L_0 v -\lambda(0) J_0 v
  $$
  defines a topological isomorphism between the spaces $\{v \in H^1_{rad}\::\: \langle v,u_0 \rangle_{L^2(\Omega)} =0\}$ and 
  $Y:= \{\phi \in (H^1_{rad})'\::\: \phi(u_0)=0\}$. From this we readily deduce (\ref{eq:claim-isomorphism}).

  From (\ref{eq:eps-solution}), (\ref{eq:claim-isomorphism}) and the simplicity of the eigenvalue ${\l(\e)}$, it follows by the implicit function theorem that the map $\eps \mapsto u_\eps$ is of class $C^1$ in a neighborhood of $\eps=0$, as claimed. 
\end{proof}

\textbf{Acknowledgements:} This work is supported by DAAD and BMBF (Germany) within the project 57385104.

\end{document}